\documentclass[10pt,reqno]{amsart}
\usepackage{amsfonts,amsmath,amssymb}
\usepackage{mathrsfs}
\usepackage{pifont}
\usepackage{enumerate}
\usepackage{amssymb,amsmath,bbm}
\usepackage{hyperref} 
\usepackage{cite} 
\usepackage{enumitem}
\setlist[enumerate]{itemsep=5pt, leftmargin=*}
\usepackage[dvipsnames]{xcolor}
\usepackage[latin1]{inputenc}

\usepackage{mathtools}
\mathtoolsset{showonlyrefs}

\makeatletter \@namedef{subjclassname@2010}{
  \textup{2010} Mathematics Subject Classification}
\makeatother

\newtheorem{thm}{Theorem}[section]

\newtheorem{prop}[thm]{Proposition}
\newtheorem{coro}[thm]{Corollary}
\newtheorem{lem}[thm]{Lemma}

\theoremstyle{definition}

\newtheorem{rem}[thm]{Remark}
\newtheorem{exem}[thm]{Example}
\newtheorem{Ques}{Question}

\newcommand{\0}{{\color{gray}0}}
\newcommand{\B}{\mathcal{B}}
\newcommand{\C}{\mathbb{C}}

\newcommand{\D}{\mathbb{D}}

\newcommand{\T}{\mathbb{T}}
\newcommand{\h}{\mathcal{H}}
\newcommand{\K}{\mathcal{K}}
\newcommand{\M}{\mathcal{M}}
\newcommand{\MM}{\mathsf{M}}
\newcommand{\N}{\mathcal{N}}
\newcommand{\R}{\mathcal{R}}
\newcommand{\Z}{\mathbb{Z}}
\newcommand{\newinner}[1]{\langle\!\langle #1 \rangle\!\rangle}
\newcommand{\newnorm}[1]{|\!|\!| #1 |\!|\!|}
\newcommand{\inner}[1]{\langle #1 \rangle}
\newcommand{\norm}[1]{\| #1 \|}
\renewcommand{\vec}[1]{\mathbf{#1}}

\newcommand{\apspec}{\sigma_{\mathrm{ap}}}
\newcommand{\pspec}{\sigma_{\mathrm{p}}}
\newcommand{\rank}{\operatorname{rank}}

\newcommand{\textminimatrix}[4]{\big[ \begin{smallmatrix} #1 & #2 \\ #3 & #4 \end{smallmatrix}\big]}

\begin{document}

\title[On $m$-partial isometries]{O\lowercase{n $m$-partial isometries: spectra, weighted shifts, and similarity}}

\author{Mohamed Amine Aouichaoui}
\address{University of Monastir, Faculty of Sciences of Monastir, 5019, Monastir, Tunisia}
\email{amine.aouichaoui@fsm.u-monastir.tn}

\author{Micha{\l} Bucha{\l}a}
\address{AGH University of Krakow, Faculty of Applied Mathematics, al. A. Mickiewicza
30, 30-059 Krakow}
\email{mbuchala@agh.edu.pl}

\author{Stephan Ramon Garcia}
\address{Department of Mathematics and Statistics, Pomona College, 610 N. College Ave.,
Claremont, CA 91711, USA}
\email{stephan.garcia@pomona.edu}
\urladdr{\url{https://stephangarcia.sites.pomona.edu/}}



\begin{abstract} 
The aim of this paper is to study $m$-partial isometries on Hilbert spaces, a natural extension of partial isometries and $m$-isometries. We establish structural and spectral results, characterize the $m$-partial isometric weighted shifts, and investigate similarity to $m$-isometries and $m$-partial isometries.
\end{abstract}

\maketitle

\section{Introduction}
The class of $m$-isometric operators was introduced by Agler in 1990 \cite{Agler}. He and Stankus investigated them further in the influential trilogy \cite{Agler1,Agler2,Agler3}. Since then $m$-isometries have been a subject of intense research. Much attention has been paid to their local properties \cite{Bermudez3,jablonski2020m} and characterizing the $m$-isometric operators in certain classes of operators \cite{Abdullah,Bermudez2, Bermudez,Gu}. There have also been several results concerning $m$-isometric dilations, which are natural generalizations of the celebrated Sz.-Nagy dilation theorem for contractions \cite{Badea3,Badea2,Buchala,Suciu3}. Generalizations of $m$-isometries have also been proposed. For example, Gleason and Richter extended the concept to commuting tuples \cite{Gleason}. Sid Ahmed and Saddi introduced $m$-partial isometries \cite{Ahmed}, which serve as a simultaneous extension of $m$-isometries and partial isometries. The aim of this paper is to study $m$-partial isometries in more detail. In particular, we investigate their spectral properties and several similarity problems.

This paper is organized as follows.
Section \ref{Section:Basics} introduces $m$-partial isometries and their basic properties.
Weighted shifts are the focus of Section \ref{Section:WeightedShifts}, in which Theorem \ref{Theorem:Structure} completely describes those that are $m$-partially isometric.
Section \ref{Section:Spectrum} concerns the spectra of $m$-partial isometries.
In particular, we prove a spectral-inclusion result (Theorem \ref{Theorem:Spectrum}) and investigate
$m$-isometric dilations of $m$-partial isometries (Theorem \ref{Theorem:Dilations}).  
Section \ref{Section:Similarity} contains several results about similarity to $m$-isometries or $m$-partial isometries.  
We conclude this paper in Section \ref{Section:Further}, in which we pose several open problems that we hope will motivate further research on these topics.

\subsection{Notation}
In what follows, $\h$ denotes a complex Hilbert space endowed with inner product $\langle \cdot \, , \, \cdot\rangle$ and associated norm $\|\cdot\|$. We denote by $\B(\h)$ the algebra of all bounded linear operators on $\h$. For each $T \in \B(\h)$, we denote by $\sigma(T)$, $\R(T)$, $\N(T)$, and $T^{*}$ the spectrum, range, kernel, and adjoint of $T$, respectively. Moreover, $\sigma(T), \sigma_{\mathrm{ap}}(T)$, and $\sigma_{\mathrm{p}}(T)$ represent the spectrum, approximate point spectrum and point spectrum of $T$, respectively. A (closed) subspace $\M\subseteq \h$ is \emph{invariant} for $T$ (or \emph{$T$-invariant}) if $T\M\subseteq \M$. The orthogonal complement of $\M$ is denoted by $\M^\perp$ and the orthogonal projection onto $\M$ by $P_{\M}$. We let $\D$, $\D^{-}$, and $\T$ denote the open unit disk, the closed unit disk, and the unit circle in the complex plane $\C$, respectively. We write $\sim$ for similarity and $\sim_{+}$ for similarity via a positive intertwiner, and we let $\MM_n$ denote the set of $n \times n$ complex matrices.

\subsection{The case $m=1$}
Recall that $T \in \B(\h)$ is an isometry if $T^{*} T=I$, in which $I$ is the identity operator. Equivalently, $T$ is a contraction with a contractive left inverse $S$; that is, there exists an $S \in \B(\h)$ with $S T=I$ and $\norm{S} \leq 1$. 
The Wold--von Neumann theorem decomposes an isometry as a direct sum of a unitary operator and a unilateral shift. 
This algebraic characterization extends up to similarity if one replaces ``contraction'' with ``bounded powers'' \cite{Nagy, Pisier}. Thus, $T$ is similar to an isometry if and only if it is power bounded and has a left inverse with bounded powers. 
In finite-dimensional spaces, isometries are unitaries.  Consequently, $A \in \MM_{n}$ is similar to an isometry if and only if it is similar to a unitary.  This occurs if and only if $A$ is diagonalizable and $\sigma(A) \subseteq \T$. 

Recall that $T \in \B(\h)$ is a partial isometry if $T T^{*} T=T$;
that is, $T^{*}$ is a generalized inverse of $T$. Mbekhta proved that $T \in \B(\h)$ is a partial isometry if and only if it is a contraction with a contractive generalized inverse; that is, an $S$ such that $STS = S$, $TST = T$, and $\norm{S} \leq 1$ \cite{Mbekhta}; see also \cite{Badea}.

\noindent\textbf{Acknowledgments.} SRG was partially supported by NSF Grant DMS-2452084. MB was supported by the subsidy granted to AGH University of Krakow by Polish Ministry of Science and Higher Education.

\section{$m$-partial isometries}\label{Section:Basics}

For $T \in \B(\h)$ and integer $m \geq 1$, define the selfadjoint operator
\begin{equation}\label{eq:Beta}
\beta_m(T) := \sum_{k=0}^m (-1)^{m-k} \binom{m}{k} \, T^{*k} T^k ,
\end{equation}
in which $\textstyle\binom{m}{k}$ is a binomial coefficient.  
If $\beta_m(T) = 0$, then $T$ is an \emph{$m$-isometry}.  
The definition $\beta_m(T) = 0$ is equivalent to
\begin{equation}\label{eq:NormCondition}
\sum_{k=0}^m (-1)^{m-k} \binom{m}{k} \norm{T^k \vec{x}}^2 = 0
\end{equation}
for all $\vec{x} \in \h$.  
Whenever we speak of an $m$-isometry it is with the understanding that $m \geq 1$. 
A $1$-isometry is an isometry in the usual sense: $T^*T = I$.
An $m$-isometry is a \emph{strict $m$-isometry} if it is not
an $(m-1)$-isometry.  Since
\begin{equation}\label{eq:BetaInduction}
\beta_{m+1}(T) = T^* \beta_m(T) T - \beta_m(T),
\end{equation}
it follows that an $m$-isometry is an $n$-isometry for all $n \geq m$.
If $T$ is an invertible $m$-isometry, then so is $T^{-1}$ because
\begin{equation}\label{eq:InverseAlso}
(-1)^m\beta_m(T^{-1}) = (T^*)^{-m}\beta_m(T) T^{-m} =0.
\end{equation}

Recall that $T \in \B(\h)$ is a \emph{partial isometry} if $TT^*T = T$;
that is, if $T\beta_1(T) = 0$.  More generally, $T$ is an \emph{$m$-partial isometry} if $T \beta_m(T) = 0$;
it is \emph{strict} if it is not an $(m-1)$-partial isometry.
These definitions originate in \cite{Ahmed}, although its consequences appear to be little pursued in the literature.

There is another characterization of partial isometries: $T\in \B(\h)$ is a partial isometry
if and only if $T^*T$ is an orthogonal projection.  In fact, $T^*T$ is the orthogonal projection onto
$\N(T)^{\perp}$.  This occurs precisely when $I - T^*T = -\beta_1(T)$
is the orthogonal projection onto $\N(T)$.  The next lemma generalizes this observation.

\begin{lem}\label{Lemma:Projection}
  Let $T\in \B(\h)$. The following are equivalent.
  \begin{enumerate}
\item $T$ is an $m$-partial isometry.
\item $(-1)^m\beta_m(T)$ is the orthogonal projection onto $\N(T)$.
  \end{enumerate}
\end{lem}

\begin{proof}
(a) $\Rightarrow$ (b). Suppose that $T$ is an $m$-partial isometry. Then $\R(\beta_{m}(T)) \subseteq \N(T)$
by definition. Moreover, $(-1)^m \beta_m(T) \vec{h} = \vec{h}$ for any $\vec{h} \in \mathcal{N}(T)$.
Therefore, $\R(\beta_m(T)) = \N(T)$.
Since $\beta_m(T)$ is selfadjoint, the definition $T \beta_m(T) = 0$ ensures that $\beta_{m}(T)T^*\vec{h} = \vec{0}$ for every $\vec{h}\in \h$. With respect to the orthogonal decomposition $\h = \N(T)\oplus \R(T^*)^-$, we have
$(-1)^m\beta_{m}(T) = \textminimatrix{I}{\0}{\0}{\0}$,
so $(-1)^m\beta_{m}(T) = P_{\N(T)}$.

\noindent (b) $\Rightarrow$ (a).  Suppose that $(-1)^m\beta_m(T) = P_{\N(T)}$.
If $\vec{h} = \vec{h}_1 + \vec{h}_2$, in which $\vec{h}_{1}\in \N(T)$ and 
$\vec{h}_{2}\in \R(T^*)^-$, then $ T\beta_{m}(T)\vec{h} = (-1)^m T P_{\N(T)} \vec{h} = (-1)^m T\vec{h}_{1} = \vec{0}$. 
Thus, $T$ is an $m$-partial isometry.
\end{proof}

\begin{rem}
    Suppose that $T \in \B(\h)$ is an $m$-partial isometry with 
    $\N(T) = \{\vec{0}\}$.  Then Lemma \ref{Lemma:Projection} ensures that $(-1)^m \beta_m(T) = 0$, so
    $T$ is an $m$-isometry.
\end{rem}

\begin{exem}
For each integer $m\geq 1$, we claim that $T=\textminimatrix{0}{0}{m^{-1/2}}{0}$ is an $m'$-partial isometry if and only if $m' = m$. Indeed,
\begin{equation*}
    (-1)^{m'}\beta_{m'}(T) = \begin{bmatrix}
    1-\frac{m'}{m} & 0\\
    0 & 0
    \end{bmatrix},
\end{equation*}
so $ (-1)^{m'}\beta_{m'}(T) = P_{\mathcal{N}(T)} $ if and only if $m' = m$. In particular, $ T $ is a strict $m$-partial isometry that is not an $(m+1)$-partial isometry.
\end{exem}

\begin{exem}
Suppose that $T \in \B(\h)$ has the block decomposition
\begin{equation*}
T=
\begin{bmatrix}
0 & A\\
0 & B
\end{bmatrix}
\end{equation*}
with respect to $\h = \N(T) \oplus \R(T^*)^-$.  Let $M=A^*A+B^*B$ and observe that
\begin{equation*}
T^*T=
\begin{bmatrix}
0 & 0\\
0 & M
\end{bmatrix}    
\quad \text{and} \quad
T^{*2}T^2=
\begin{bmatrix}
0 & 0\\
0 & B^*MB
\end{bmatrix},
\end{equation*}
so
\begin{equation}\label{eq:TTB}
T^{*2}T^2-2T^*T+I
=
\begin{bmatrix}
I & 0\\
0 & B^*MB-2M+I
\end{bmatrix}.    
\end{equation}
Thus, $T$ is a $2$-partial isometry if and only if
$T(T^{*2}T^2-2T^*T+I)=0$; that is,
\begin{equation*}
A(B^*MB-2M+I)=0
\quad \text{and} \quad
B(B^*MB-2M+I)=0.
\end{equation*}
The condition $B^*MB-2M+I=0$ ensures that $T$ is a $2$-partial isometry. If this occurs, 
\eqref{eq:TTB} is not $0$, so $T$ is not a $2$-isometry, whereas
$T(T^{*2}T^2-2T^*T+I)=0$, so $T$ is a $2$-partial isometry.
Such block operator matrices appear in \cite{Fialkow}, where the authors studied conditions which may guarantee similarity to a partial isometry.
\end{exem}

It is well-known that the adjoint of any partial isometry is again a partial isometry. 
However, this property does not hold in general for $m$-partial isometries.

\begin{exem}\label{ex:m-partial-adjoint}
A computation confirms that 
\begin{equation*}
T = 
\begin{bmatrix}
0 & 0 & 0 \\
\sqrt{ \frac{2}{3}} & 0 & 0 \\
0 & \frac{\sqrt{2}}{2} & 0
\end{bmatrix} \in \B(\C^3)    
\end{equation*}
is a $2$-partial isometry and $T^*$ is not.
\end{exem}

However, there are certain conditions that ensure that the adjoint of an $m$-partial isometry is an $m$-partial isometry.

\begin{rem}
We claim that if $T\in \B(\C^2)$ is a $2$-partial isometry, then $T^*$ is a $2$-partial isometry.
If $\rank T = 0$, the assertion is trivial. 
If $\rank T = 2$, then $T$ is invertible and hence $T$ is a $2$-isometry.
Theorem 2.7 in \cite{Bermudez} ensures that $T$ is unitary, so $T^*$ is a $2$-partial isometry.
Now suppose that $\rank T = 1$.  Then there are unit vectors $\vec{u},\vec{v}\in \C^2$ and a scalar $\alpha > 0$ such that $T\vec{x} = \alpha \inner{ \vec{x}, \vec{v} } \vec{u}$ for all $\vec{x}\in \C^2$.  Moreover, $\N(T) = \operatorname{span}\{ \vec{v} \}^{\perp}$,
$T\vec{v}=\alpha \vec{u}$, and $T^2 \vec{v}=\alpha^2\inner{ \vec{u}, \vec{v} } \vec{u}$, so
the $2$-partial isometry condition is
$\alpha^4 |\langle u,v\rangle|^2-2\alpha^2+1=0$.    
Similarly, $T^*\vec{x}=\alpha \inner{ \vec{x},\vec{u} }\vec{v}$ for all $\vec{x} \in \C^2$
and $\N(T^*)^{\perp}=\operatorname{span}\{u\}$.  Then $T^*\vec{u}=\alpha \vec{v}$ and
$(T^*)^2 \vec{u} =\alpha^2\inner{ \vec{v},\vec{u} } \vec{v}$. Since
$\alpha^4 |\langle v,u\rangle|^2-2\alpha^2+1=0$, we conclude that $T^*$ is a $2$-partial isometry.
\end{rem}

\begin{rem}
Recall that $T \in \B(\h)$ is \emph{complex symmetric} if there is a conjugation $C$
(a conjugate-linear, isometric involution) such that $T = CT^*C$ \cite{SNCSO, MPCSO, CSOA, CSO2}.
If $T$ is a complex symmetric $m$-partial isometry, then $T^*$ is too since $C (T^{*k}T^k)C = T^k T^{*k}$.
\end{rem}

If $ T\in \B(\h) $ is a partial isometry, then $TT^*T = T$; that is, $T^*$ 
is a generalized inverse of $T$. The lemma below generalizes this to $m$-partial isometries.

\begin{lem}\label{lem:pseudompartial}
    Let $ T\in \B(\h) $ be an $m$-partial isometry. Then for
    \begin{equation*}
        S = \sum_{k=1}^{m} (-1)^{k-1} \binom{m}{k} T^{\ast k}T^{k-1}
    \end{equation*}
    we have $ T = TST $ and $ S = STS $.    
\end{lem}
\begin{proof}
    First, note that
    \begin{align*}
        ST &= \sum_{k=1}^{m} (-1)^{-k+1} \binom{m}{k} T^{\ast k}T^{k}\\
        &= (-1)^{m+1}\sum_{k=1}^{m} (-1)^{m-k} \binom{m}{k} T^{\ast k}T^{k}\\
        &= (-1)^{m+1}\left( \beta_m(T)-(-1)^mI\right)\\
        &= I-(-1)^m\beta_m(T).
    \end{align*}
    Then Lemma \ref{Lemma:Projection} ensures that $ST$ is the orthogonal projection onto $\R(T^*)^-$. Then
    \begin{equation*}
        T = T(P_{\N(T)}+P_{\R(T^*)^-}) = TP_{\R(T^*)^-} = TST.
    \end{equation*}
    In turn, since $\R(S) \subset \R(T^*) $ it follows that $STS = S$.
\end{proof}

\section{Weighted shifts}\label{Section:WeightedShifts}

Berm{\'u}dez, Martin\'on, and Negr\'in characterized the $m$-isometric weighted shifts in 2010
\cite{Bermudez2}. A polynomial description of their weight sequences was
later obtained by Abdullah and Le \cite{Abdullah}. 
In this section we study the corresponding problems for $m$-partial
isometries. We completely describe $m$-partially isometric weighted
shifts and illustrate the results with several examples.
We begin with a computation that translates the equation
$T\beta_m(T)=0$ into a recurrence for the weights.

\begin{exem}
Let $(e_{n})_{n \geq 1}$ be an orthonormal basis of $\h$. For a numerical sequence $(\omega_{n})_{n \geq 1}$, the associated \emph{weighted shift operator} $T$ on $\h$ is defined by
$T e_{n}=\omega_{n} e_{n+1}$ for $n \geq 1$.  It is bounded if and only if the weight sequence is bounded;
we assume this is the case in what follows.  One can show that
\begin{equation*}
T^{k} e_{n} = \bigg( \prod_{j=n}^{n+k-1} \omega_{j} \bigg) e_{n+k}
\end{equation*}
and
\begin{equation*}
T^{* k} e_n =
\begin{cases}
0, & n\leq k,\\[3pt]
\bigg(\displaystyle\prod_{j=n-k}^{n-1}\overline{\omega_j}\bigg)e_{n-k},
& n>k.
\end{cases}  
\end{equation*}
Thus,
\begin{equation*}
T^{* k} T^{k} e_{n} = \bigg( \prod_{j=n}^{n+k-1} | \omega_{j} |^2 \bigg) e_{n}
\end{equation*}
and hence $T$ is an $m$-partial isometry if and only if the following holds for all $n \geq 1$:\begin{equation}\label{eq:WeightedCondition}
\omega_{n} \bigg[ (-1)^{m} + \sum_{1 \leq k \leq m} (-1)^{m-k} \binom{m}{k} \bigg( \prod_{j=n}^{n+k-1} | \omega_{j} |^2 \bigg) \bigg] = 0.    
\end{equation}
\end{exem}

\begin{rem}
Let $a_n=|\omega_n|^2$ for $n \geq 1$.
If $a_n,a_{n+1},\ldots,a_{n+m-2}\neq 0$, the relation
\begin{equation*}
(-1)^m + \sum_{k=1}^{m} (-1)^{m-k}\binom{m}{k} \prod_{j=n}^{n+k-1}a_j =0
\end{equation*}
ensures that
\begin{equation*}
a_{n+m-1}
= \frac{ (-1)^{m+1} - \displaystyle\sum_{k=1}^{m-1}
(-1)^{m-k}\binom{m}{k} \prod_{j=n}^{n+k-1}a_j }{ \displaystyle\prod_{j=n}^{n+m-2}a_j }.    
\end{equation*}
Thus, admissible blocks of nonzero weights may be constructed recursively, so long as the right side is positive. If the right side is $0$, then $a_{n+m-1}=0$, so the string of nonzero weights terminates. Negative values are not admissible, since $a_j=|\omega_j|^2\geq 0$, although the recursion can produce such values.  If $m=2$, then
\begin{equation*}
a_{n+1}
= \frac{-1-(-1)\binom{2}{1}a_n }{ a_n }
= \frac{-1+2a_n}{a_n}
= 2-\frac{1}{a_n}.
\end{equation*}
Thus, $a_{n+1}\geq 0$ if and only if $a_n \geq \frac{1}{2}$.  For example, if $a_n = \frac{1}{4}$, then
$a_{n+1} = -2<0$, which is inadmissible.
\end{rem}

The next theorem characterizes $m$-partially isometric weighted shifts.

\begin{thm}\label{Theorem:Structure}
Let $T\in \B(\h)$ be an $m$-partially isometric
weighted shift with weights $(\omega_n)_{n\geq 1}$.
Let $Z=\{n\geq 1:\omega_n=0\}$.
\begin{enumerate}
\item If $Z=\varnothing$, then $T$ is an $m$-isometry.

\item If $Z\neq\varnothing$, then write $Z=\{n_r:r\in I\}$, in which $I$ is an index set and 
$n_1 < n_2 < \cdots $ is an increasing sequence (finite or infinite).
Let $n_0=0$ and
\begin{equation*}
\h_r = \operatorname{span}\{e_{n_{r-1}+1},e_{n_{r-1}+2},\ldots,e_{n_r}\}
\end{equation*}
for each $r \in I$.  If $I=\{1,2,\ldots,s\}$ is finite, let
$\h_{\infty} =  \overline{\operatorname{span}\{e_{n_s+1},e_{n_s+2},\ldots\}}$.
If $I=\mathbb{N}$, then omit $\h_{\infty}$ in what follows. Then
\begin{equation}\label{eq:OrthogonalDecomposition}
\h = \bigg(\bigoplus_{r\in I}\h_r\bigg) \oplus \h_{\infty}    
\end{equation}
is an orthogonal decomposition into reducing subspaces for $T$. With respect to this decomposition,
\begin{equation*}
T =
\bigg(\bigoplus_{r\in I}N_r\bigg)\oplus S,    
\end{equation*}
in which each $N_r = T|_{\h_r}$ is nilpotent of order
$d_r=n_r-n_{r-1}$.
Moreover, $S=T|_{\h_{\infty}}$ is an $m$-isometric weighted shift unless $I = \mathbb{N}$,
in which case it is omitted.
\end{enumerate}
\end{thm}

\begin{proof}
Suppose that $Z=\varnothing$. Then $\omega_n\neq 0$ for all $n \geq 1$, so \eqref{eq:WeightedCondition} ensures that 
\begin{equation*}
(-1)^m + \sum_{k=1}^{m} (-1)^{m-k}\binom{m}{k} \prod_{j=n}^{n+k-1}a_j =0    
\end{equation*}
for all $n\geq 1$.  Equivalently, $\beta_m(T)\vec{e}_n=0$ for
all $n \geq 1$, so $T$ is an $m$-isometry.

Suppose that $Z \neq \varnothing$. For each $r\in I$, let
$\K_r =\operatorname{span}\{\vec{e}_1,\vec{e}_2,\ldots,\vec{e}_{n_r}\}$. Since $\omega_{n_r}=0$, it follows that
$T\vec{e}_j=\omega_j \vec{e}_{j+1}\in \K_r$ for all $1\leq j\leq n_r$. Thus, $\K_r$ is $T$-invariant. Since
\begin{equation*}
T^*\vec{e}_j=
\begin{cases}
0 & \text{if $j=1$},\\[3pt]
\overline{\omega_{j-1}} \vec{e}_{j-1} & \text{if $j\geq 2$},
\end{cases}    
\end{equation*}
we see that $T^*\vec{e}_j\in \K_r$ for all $1\leq j\leq n_r$.
Thus, each $\K_r$ reduces $T$ and, since they are nested,
their orthogonal differences $\h_r=\K_r\ominus \K_{r-1}
=
\operatorname{span}\{\vec{e}_{n_{r-1}+1},\ldots,\vec{e}_{n_r}\}$
also reduce $T$. 
If $I=\{1,2,\ldots,s\}$, then $\h_{\infty}=\K_s^\perp$
reduces $T$. This yields \eqref{eq:OrthogonalDecomposition}.

Since $n_{r-1}$ and $n_r$ are consecutive zero positions
in the weight sequence, it follows that
$\omega_{n_{r-1}+1},\ldots,\omega_{n_r-1}\neq 0$.
On $\h_r$, the restriction $N_r =T|_{\h_r}$ acts as follows:
\begin{equation*}
\vec{e}_{n_{r-1}+1}
\mapsto
\omega_{n_{r-1}+1} \vec{e}_{n_{r-1}+2}
\mapsto
\cdots
\mapsto
\omega_{n_r-1} \vec{e}_{n_r}
\mapsto
0.    
\end{equation*}
Thus, $N_r^{d_r}=0$, in which $d_r=n_r-n_{r-1}$.
Since an empty product equals $1$, 
\begin{equation*}
N_r^{d_r-1}e_{n_{r-1}+1}
= \bigg( \prod_{j=n_{r-1}+1}^{n_r-1}\omega_j \bigg)e_{n_r} \neq 0,
\end{equation*}
so $N_r$ is nilpotent of order $d_r$.

If $I=\mathbb{N}$, there is nothing to prove. Suppose 
that $I=\{1,2,\ldots,s\}$. By construction, $\omega_n\neq 0$ for $n>n_s$ and hence the factor $\omega_n$ in
\eqref{eq:WeightedCondition} is nonzero, so
\begin{equation}\label{eq:WeightTail}
(-1)^m + \sum_{k=1}^{m} (-1)^{m-k}\binom{m}{k} \prod_{j=n}^{n+k-1}a_j =0.    
\end{equation}
With respect to the orthonormal basis
$\vec{e}_{n_s+1},\vec{e}_{n_s+2},\ldots$ of $\h_{\infty}$,
the restriction $S=T|_{\h_{\infty}}$
is a weighted shift with nonzero weights.  The identity
\eqref{eq:WeightedCondition} ensures that $S$ is an $m$-isometry.
\end{proof}

\begin{exem}
For each $m\geq 1$, we construct an $m$-partially isometric weighted shift $T$ that is similar to $S \otimes S^*$, the tensor product of the unilateral shift $S$ with its adjoint. First recall that $S \otimes S^*$ is unitarily equivalent to $\bigoplus_{n=1}^{\infty} J_n$, in which $J_n$ denotes the $n \times n$ nilpotent Jordan block \cite[Thm.~15.4.9]{OTBE}.  
Fix $m \geq 1$ and place a $0$ weight at each triangular number; that is, at positions $1,3,6,10,15,\ldots$.
This leaves room for separate stretches of $\ell = 1,2,3,4,\ldots$ positive weights between zeros.  To fill
a stretch of $\ell\geq 1$ positive weights use the sequence
\begin{equation*}
\sqrt{\frac{\ell}{\ell+m-1}},\,  
\sqrt{\frac{\ell-1}{\ell+m-2}},\,  
\sqrt{\frac{\ell-2}{\ell+m-3}},\,  
\ldots,\,  
\sqrt{\frac{2}{m+1}},\,  
\sqrt{\frac{1}{m}},
\end{equation*}
each element of which belongs to $[2^{1-m},1)$; in particular, the weights are bounded above and below.  Thus, the weight sequence is
\begin{equation*}
    0,\,  
\sqrt{\frac{1}{m}},\,  
0,\,  
\sqrt{\frac{2}{m+1}},\, \sqrt{\frac{1}{m}},\,  
0,\,  
\sqrt{\frac{3}{m+2}},\, \sqrt{\frac{2}{m+1}},\, \sqrt{\frac{1}{m}},\,  
0,\ldots .
\end{equation*}
The weighted shift $T$ so constructed is a bounded $m$-partial isometry and,
moreover, it is similar to
$\bigoplus_{n=1}^{\infty} J_n$, in which $J_n$ denotes the $n \times n$ nilpotent Jordan block.
\end{exem}

\begin{exem}
Let $\omega_2 = 0$ and $\omega_n = 1$ for all $n \neq 2$.
Then, the associated weighted shift $T$ is a $1$-partial isometry since
the condition $T\beta_1(T)=0$ demands that $\omega_n(|\omega_n|^2-1)=0$ for all $n \geq 1$.
However, $T$ is not a $1$-isometry (that is, an isometry) since it has nontrivial kernel.
More generally, for every $m\geq 1$, 
\begin{equation*}
\omega_1=\frac{1}{\sqrt m},\qquad \omega_2=0,\quad \text{and} \quad \omega_n=1\quad (n\geq 3)    
\end{equation*}
defines an $m$-partial isometry that is neither nilpotent nor an $m$-isometry.
\end{exem}

\begin{exem}
A nilpotent weighted shift need not be an $m$-partial isometry.
For example, let $\omega_1=2$ and $\omega_n=0$ for $n\geq 2$.
Then $T^2=0$ and $T$ is not an $m$-partial isometry for any $m\geq 1$.
Indeed, apply \eqref{eq:WeightedCondition} with $n=1$ and use the fact
that all products of length at least $2$ contain $\omega_2=0$ to obtain
\begin{equation*}
\omega_1 \bigg[ (-1)^m+(-1)^{m-1}\binom{m}{1}|\omega_1|^2 \bigg]=0.    
\end{equation*}
Since $\omega_1=2\neq 0$, this is equivalent to
$(-1)^m+(-1)^{m-1}m|\omega_1|^2=0$,
or, equivalently,
$4= |\omega_1|^2=\frac{1}{m}$, which is impossible.
\end{exem}

\begin{rem}
To obtain an $m$-partial isometry that is not an $m$-isometry, it is necessary that at least one weight vanish.
If no weight vanishes, then $a_n \to 1$ since a weighted shift is $m$-isometric if and only if $a_n = p(n+1)/p(n)$ for some polynomial $p$ of degree at most $m-1$ \cite[Thm.~2.1]{Abdullah}.
\end{rem}

\begin{exem}
Let $m=2$. Then \eqref{eq:WeightedCondition} says that
$\omega_n(1-2a_n+a_na_{n+1})=0$ for all $n \geq 1$.  Thus, 
$a_{n+1} = 2 - 1/a_n$ whenever $\omega_n\neq 0$;
in particular, observe that $a_n = 1$ if and only if $a_{n+1}=1$.
If $a_1 = 1$, then $a_n = 1$ for all $n \geq 1$, in which case $T$ is an isometry.
Now suppose that $a_1 \neq 1$.  Then the sequence $b_n = 1/(a_n-1)$ satisfies
\begin{equation*}
b_{n+1}
=
\frac{1}{a_{n+1}-1}
=
\frac{1}{1-\frac{1}{a_n}}
=
\frac{a_n}{a_n-1}
=
1+\frac{1}{a_n-1}
=
b_n+1,    
\end{equation*}
so $b_n=b_1+n-1$ and hence
\begin{equation*}
a_n = 1+\frac{1}{n-1+\frac{1}{a_1-1}}
\end{equation*}
is nonnegative and tends to $1$; in particular, $\omega_n$ is a bounded sequence.  
Moreover, $a_n = 0$ if and only if $n = -1/(a_1-1)$;
for example, $a_1 = \frac{3}{4}$ yields $a_2 = \frac{2}{3}$, $a_3 = \frac{1}{2}$,
and $a_4 = 0$.  Note that $T$ is not a $2$-isometry since
$1-2a_4+a_4a_5=1\neq 0$.  Thus, $T$  is a $2$-partial isometry that is not a $2$-isometry.
Note that $T$ is a $2$-isometry 
if $a_1 \neq 1 - \frac{1}{n}$ for any integer $n \geq 1$.
\end{exem}

\section{Spectrum}\label{Section:Spectrum}
In this section we study the spectrum of $m$-partial isometries.  In particular, we prove a sharp
spectral-inclusion theorem. We begin with the following lemma.

\begin{lem}\label{boundarycompact}
Let $K \subseteq \C$ be compact and $r>0$. If $\partial K\subseteq r\D^-$, then $K\subseteq r\D^-$.
\end{lem}

\begin{proof}
Since $K$ is compact, the continuous function $z\mapsto |z|$ attains its maximum on $K$. 
Thus, there exists a $w\in K$ such that $|w|=\max_{z\in K}|z|$.
If $w \in \operatorname{int} K$, then $K$ contains points of
modulus greater than $|w|$, which is impossible.  Thus, $w \in \partial K \subseteq r \D^-$
and hence each $z \in K$ satisfies $|z| \leq |w| \leq r$, so $K \subseteq r\D^-$.
\end{proof}

The computation below owes much to \cite[Lem.~1.21]{Agler1} and \cite[Prop.~2.1]{Cho}.

\begin{thm}\label{Theorem:Spectrum}
Let $T\in \B(\h) $ be an $m$-partial isometry.
\begin{enumerate}
\item For $m$ odd, $\sigma(T) \subseteq \D^-$. In particular, $ r(T) \leq 1.$
\item For $m$ even, $\sigma(T)\subseteq (\sqrt{2}\,\D)^-$ and $\pspec(T) \subseteq \sqrt{2}\,\D$. In particular, $ r(T) \leq \sqrt{2}.$
\end{enumerate}
\end{thm}

\begin{proof}
Since $\sigma(T)$ is compact and $\partial\sigma(T)\subseteq \sigma_{\mathrm{ap}}(T)$,
Lemma~\ref{boundarycompact} says that it suffices to establish the desired inclusion for the approximate point spectrum.

Let $\lambda \in \apspec(T)$ and let $\vec{x}_n$ be unit vectors such that
$T\vec{x}_n = \lambda \vec{x}_n + \vec{o}_n(1)$, in which $\vec{o}_n(1)$ is a sequence of vectors that tends to $\vec{0}$ as $n \to \infty$.  The boundedness of $T$ ensures that $T^k \vec{x}_n = \lambda^k \vec{x}_n + \vec{o}_n(1)$ for each integer $k\geq 1$.  Therefore,
\begin{align}
1 
&= \norm{ \vec{x}_n}^2 
\geq \norm{ P_{\N(T)} \vec{x}_n }^2 
= \inner{ (-1)^m \beta_{m}(T)\vec{x}_n,\vec{x}_n} \label{eq:NonNegP}\\
&= \sum_{k=0}^m (-1)^k \binom{m}{k} \inner{ T^{*k}T^k\vec{x}_n, \vec{x}_n} 
= \sum_{k=0}^m (-1)^k \binom{m}{k} \norm{T^k \vec{x}_n}^2 \nonumber \\
&= \sum_{k=0}^m (-1)^k \binom{m}{k} |\lambda|^{2k} + o_n(1)
= (1 - |\lambda|^2)^m + o_n(1), \nonumber
\end{align}
in which $o_n(1)$ is a scalar sequence tending to $0$.  We conclude that 
\begin{equation*}
0 \leq (1 - |\lambda|^2)^m \leq 1,    
\end{equation*}
the lower bound coming from the nonnegativity of $\norm{P_{\N(t)}\vec{x}_n}$ in \eqref{eq:NonNegP}.
If $m$ is odd, then $|\lambda| \leq 1$.
If $m$ is even, then $(1 - |\lambda|^2) \in [-1,1]$ and hence $|\lambda|\leq \sqrt{2}$.
Since $\pspec(T) \subseteq \apspec(T)$, it suffices to show that $T$ has no eigenvalues
of modulus $\sqrt{2}$.  Suppose toward a contradiction that $T \vec{x} = \lambda \vec{x}$, in which
$|\lambda|=\sqrt{2}$ and $\vec{x}$ is a unit vector. Then \eqref{eq:NonNegP} ensures that $1 = \norm{P_{\N(T)}\vec{x}}^2$,
so $\vec{x} \in \N(T)$, a contradiction.
\end{proof}

\begin{exem}
    Let $m$ be odd and $0 < a < 1$.  Then $T = \textminimatrix{a}{0}{b}{0}$, in which
    \begin{equation*}
        b= a\sqrt{ \frac{(1-a^2)^m}{1 - (1-a^2)^m } }
    \end{equation*} 
    is a strict $m$-partial isometry with $\pspec(T) = \{0,a\}$.
    Thus, Theorem \ref{Theorem:Spectrum}.a is sharp for all odd $m$.  
\end{exem}

\begin{exem}
    Let $m$ be even and $0 < a < \sqrt{2}$.  Then $T = \textminimatrix{a}{0}{b}{0}$, in which
    \begin{equation*}
        b = 
        \frac{a (a^2-1)^{m/2}}{ \sqrt{ 1 - (a^2-1)^m } },
    \end{equation*}
    is a strict $m$-partial isometry with $\pspec(T) = \{0,a\}$.
    Thus, Theorem \ref{Theorem:Spectrum}.b is sharp for all even $m$.  
\end{exem}

\begin{rem}
Any compact subset of $\D^-$ (respectively, $(\sqrt{2}\,\D)^-$) that contains $0$ can be the spectrum of an
$m$-partial isometry with $m$ odd (respectively, even).  Simply take direct sums of the matrices in
the previous examples.
\end{rem}

Theorem~\ref{Theorem:Spectrum} gives spectral bounds for $m$-partial isometries. We next record two boundary consequences. 

\begin{coro}\label{cor:outer-boundary-ap}
Let $T\in\B(\h)$ be an $m$-partial isometry.
\begin{enumerate}
\item If $m$ is odd, then $\sigma(T)\cap\T=\apspec(T)\cap\T$.
\item If $m$ is even, then $\sigma(T)\cap\sqrt{2}\,\T=\apspec(T)\cap\sqrt{2}\,\T$
and $\pspec(T)\cap\sqrt{2}\,\T=\varnothing$.
\end{enumerate}
\end{coro}

\begin{proof}
We use the standard inclusion
$\partial\sigma(T)\subseteq\apspec(T)$. 

\noindent(a) If $m$ is odd, then
Theorem~\ref{Theorem:Spectrum} gives $\sigma(T)\subseteq\D^-$. Hence every point of $\sigma(T)\cap\T$ belongs to $\partial\sigma(T)$, and therefore to
$\apspec(T)$. The reverse inclusion is automatic. 

\noindent(b) If $m$ is even, then Theorem~\ref{Theorem:Spectrum} gives $\sigma(T)\subseteq\sqrt{2}\,\D^-$. Thus, every point of
$\sigma(T)\cap\sqrt{2}\,\T$ belongs to $\partial\sigma(T)$, and hence to
$\apspec(T)$. The reverse inclusion is automatic. Finally,
Theorem~\ref{Theorem:Spectrum} also gives $\pspec(T)\subseteq\sqrt{2}\,\D$, so
$\pspec(T)\cap\sqrt{2}\,\T=\varnothing$.
\end{proof}

\begin{prop}\label{prop:unit-circle}
Let $T\in\B(\h)$ be an $m$-partial isometry. 
\begin{enumerate}
\item If $\lambda\in\apspec(T)\cap\T$, then $\overline\lambda\in\apspec(T^*)$.
\item If $\lambda\in\pspec(T)\cap\T$, then $\overline\lambda\in\pspec(T^*)$.
\item If $T\vec{x}=\lambda \vec{x}$ and $T\vec{y}=\mu \vec{y}$ with  $\lambda\neq\mu$, and $|\lambda|=1$ or $|\mu|=1$, then $\vec{x}\perp \vec{y}$.
\item Let $\lambda\in\apspec(T)\cap\T$ and $\mu\in\apspec(T)$ with $\lambda\neq\mu$. If $\vec{x}_n$ and $\vec{y}_n$ are unit approximate eigensequences for $\lambda$ and $\mu$, respectively, then $\langle{\vec{x}_n},{\vec{y}_n}\rangle\to0$.
\end{enumerate}
\end{prop}

\begin{proof}
(a) Let $\lambda\in\apspec(T)\cap\T$ and choose unit vectors $\vec{x}_n$ such that $(T-\lambda I)\vec{x}_n\to0$. The proof of Theorem \ref{Theorem:Spectrum} shows that
\begin{equation*}
\norm{P_{\N(T)}\vec{x}_n}^2
=(1-|\lambda|^2)^m+\vec{o}_n(1)=\vec{o}_n(1).
\end{equation*}
Consequently, $\beta_m(T)\vec{x}_n=(-1)^m P_{\N(T)}\vec{x}_n\to0$. On the other hand, since $T^kx_n=\lambda^kx_n+\vec{o}_n(1)$ for each fixed $k$, we have
\begin{align*}
\beta_m(T)\vec{x}_n
&=\sum_{k=0}^{m}(-1)^{m-k}\binom{m}{k}T^{*k}T^k \vec{x}_n\\
&=\sum_{k=0}^{m}(-1)^{m-k}\binom{m}{k}\lambda^kT^{*k} \vec{x}_n+\vec{o}_n(1) \\
&=(\lambda T^*-I)^m \vec{x}_n+\vec{o}_n(1).
\end{align*}
Thus, $(\lambda T^*-I)^m \vec{x}_n\to0$. Since $|\lambda|=1$, we have $\lambda T^*-I=\lambda(T^*-\overline\lambda I)$, so
\begin{equation*}
(T^*-\overline\lambda I)^m\vec{x}_n\to0.
\end{equation*}
If $T^*-\overline\lambda I$ were bounded below, then so would its $m$th power be, which is impossible because $\norm{\vec{x}_n}=1$. Therefore, $\overline\lambda\in\apspec(T^*)$.

\noindent(b) Suppose that $\lambda\in\pspec(T)\cap\T$ and $T\vec{x}=\lambda \vec{x}$ with $\vec{x}\neq \vec{0}$. Then 
\begin{equation*}
\norm{P_{\N(T)}\vec{x}}^2=(1-|\lambda|^2)^m\norm{\vec{x}}^2=0,
\end{equation*} 
so $P_{\N(T)}\vec{x}=0$ and hence $\beta_m(T)\vec{x}=0$. As above, but with an exact eigenvector,
\begin{equation*}
\beta_m(T)\vec{x}=(\lambda T^*-I)^m\vec{x}.
\end{equation*}
Thus, $(T^*-\overline\lambda I)^m \vec{x}=\vec{0}$. Let $q\geq1$ be minimal such that $(T^*-\overline\lambda I)^q \vec{x}=\vec{0}$. Then $\vec{y}=(T^*-\overline\lambda I)^{q-1} \vec{x} \neq \vec{0}$
and $(T^*-\overline\lambda I)\vec{y}=0$, so $\overline\lambda\in\pspec(T^*)$.

\noindent(c) Without loss of generality, suppose that $|\lambda|=1$. Then $P_{\N(T)}\vec{x}= \vec{0}$ and hence Lemma \ref{Lemma:Projection} yields
\begin{align*}
0
&=\langle P_{\N(T)}\vec{x}, \vec{y}\rangle
=\sum_{k=0}^{m}(-1)^k\binom{m}{k}\inner{T^k \vec{x} , T^k\vec{y}}\\
&=\sum_{k=0}^{m}(-1)^k\binom{m}{k}(\lambda\overline\mu)^k\inner{\vec{x},\vec{y}}
=(1-\lambda\overline\mu)^m\langle \vec{x},\vec{y}\rangle.
\end{align*}
Since $|\lambda|=1$ and $\lambda\neq\mu$, we get $\langle \vec{x}, \vec{y} \rangle=0$. 

\noindent(d) Let $\lambda\in\apspec(T)\cap\T$ and $\mu\in\apspec(T)$ with $\lambda\neq\mu$. Let $\vec{x}_n$ and $\vec{y}_n$ be unit approximate eigensequences for $\lambda$ and $\mu$. As above, $\norm{P_{\N(T)}\vec{x}_n}\to 0$ and
\begin{equation*}
\langle {P_{\N(T)}\vec{x}_n, \vec{y}_n}\rangle
=(1-\lambda\overline\mu)^m\langle \vec{x}_n,\vec{y}_n \rangle+\vec{o}_n(1).
\end{equation*}
The left side tends to $0$ and $1-\lambda\overline\mu\neq0$ since $|\lambda|=1$ and $\lambda\neq\mu$, so $\langle \vec{x}_n,\vec{y}_n\rangle\to0$.
\end{proof}

\begin{coro}\label{cor:unit-circle}
Suppose that $T$ and $T^*$ are $m$-partial isometries. Then
\begin{equation*}
\lambda\in\apspec(T)\cap\T
\quad\Longleftrightarrow\quad
\overline\lambda\in\apspec(T^*)\cap\T
\end{equation*}
and
\begin{equation*}
\lambda\in\pspec(T)\cap\T
\quad\Longleftrightarrow\quad
\overline\lambda\in\pspec(T^*)\cap\T.
\end{equation*}
Consequently,
\begin{equation*}
\sigma(T)\cap\T=\apspec(T)\cap\T
\quad\text{and}\quad
\sigma(T^*)\cap\T=\apspec(T^*)\cap\T.
\end{equation*}
\end{coro}

\begin{proof} 
The equivalences follow by applying Proposition~\ref{prop:unit-circle} to $T$ and $T^*$. We prove the spectral equality for $T$; the proof for $T^*$ is identical. We use the decomposition
\begin{equation}\label{eq:spectrum-decomposition}
\sigma(T)=\apspec(T)\cup\{\lambda\in\C:\overline\lambda\in\pspec(T^*)\}.
\end{equation}
Let $\lambda\in\sigma(T)\cap\T$. If $\lambda\notin\apspec(T)$, then \eqref{eq:spectrum-decomposition} implies that $\overline\lambda\in\pspec(T^*)\cap\T$. By Proposition \ref{prop:unit-circle}, $\lambda\in\pspec(T)\subseteq\apspec(T)$, a contradiction. Hence $\sigma(T)\cap\T\subseteq\apspec(T)\cap\T$. The reverse inclusion is automatic.
\end{proof}

\begin{rem} 
Proposition~\ref{prop:unit-circle} and Corollary~\ref{cor:unit-circle} are reduction-free analogues of \cite[Prop.~3.9, Cor.~3.2]{Ahmed}. Their results assume that the relevant kernels are reducing subspaces. Here, no reducing hypothesis is imposed. Under the reducing hypothesis, the nonzero approximate spectral values are
forced onto $\T$, so Proposition~\ref{prop:unit-circle} and Corollary~\ref{cor:unit-circle} recover the
corresponding conclusions of \cite[Prop.~3.9, Cor.~3.2]{Ahmed}. \end{rem}

Since partial isometries have isometric dilations, it is natural to ask whether $m$-partial isometries possess $m$-isometric dilations;
see Question \ref{Question:Dilation} in Section \ref{Section:Further}.

\begin{thm}\label{Theorem:Dilations}
    Let $ T\in \mathcal{B}(\h) $ be an $m$-partial isometry. 
    \begin{enumerate}
        \item If $m$ is even and $r(T) > 1$, then $T$ has no $m'$-isometric dilation for any $m'$.
        \item If $m = 3$, then $T$ has a $3$-isometric dilation.
    \end{enumerate}
\end{thm}

\begin{proof}
    (a) If $T$ has an $m'$-isometric dilation, then $\|T^n\|^2/ n^{m'-1}$ is bounded as $n\to \infty$ \cite[p.~389]{Agler1}. 
    However, the spectral-radius formula ensures that
    $1<r(T) \leq \| T^n\|^{1/n}$ for all $n \geq 1$,
    so the powers of $T$ grow exponentially.

    \noindent(b) Lemma \ref{Lemma:Projection} implies that $\beta_3(T) \leq 0$ and \cite[Thm.~3.6]{Buchala}
    ensures that $T$ has a $3$-isometric dilation.
\end{proof}

\section{Similarity}\label{Section:Similarity}

We turn our attention to the problem of similarity to an $m$-isometry or $m$-partial isometry. 
The next result proves that the challenge of similarity to an $m$-isometry (or an $m$-partial isometry) can be interpreted as a renormalization problem.

\begin{thm}\label{ThmSimilarToMIsometryRenorming}
An operator $T \in \B(\h)$ is similar to an $m$-isometry (respectively, $m$-partial isometry) if and only if there exists an equivalent Hilbert norm such that $T$ is an $m$-isometry (respectively, $m$-partial isometry) with respect to this norm.
\end{thm}

\begin{proof}
Suppose that $T=P^{-1} S P$, in which $P \in \B(\h)$ is invertible and $S$ is an $m$-isometry (respectively, $m$-partial isometry). Then $\newnorm{\vec{x}} = \langle P^{*} P x, x \rangle^{1 / 2}$ is an equivalent Hilbert norm and $T$ is an $m$-isometry (respectively, $m$-partial isometry) with respect to this norm.
For the converse, suppose that $\newnorm{\,\cdot\,}$ is a Hilbert norm with inner product $\newinner{\, \cdot\, ,\, \cdot\,}$. Then there is a positive, invertible $A \in \B(\h)$ such that
$\newinner{ \vec{x}, \vec{y} } = \langle A \vec{x}, \vec{y} \rangle$.  Let $P = A^{1/2}$.  Then
for any $\vec{x} \in \h$ (respectively, for any $\vec{x} \in \N(T)^{\perp_{\textrm{New}}}=P^{-1}((P\mathcal N(T))^\perp)=P^{-1}(\mathcal N(PTP^{-1})^\perp))$,
\begin{equation*}
\sum_{k=0}^{m} (-1)^{k} \binom{m}{k} \, \newnorm{ T^{m-k} \vec{x}}^2 = 0
\iff
\sum_{k=0}^{m} (-1)^{k} \binom{m}{k} \, \|P T^{m-k} \vec{x}\|^2 = 0 .    
\end{equation*}
Set $\vec{x} = P^{-1} \vec{y}$ in the last equality and 
we conclude that $T$ is similar to an $m$-isometry (respectively, $m$-partial isometry).
\end{proof}

The spectrum of an invertible $m$-isometry is contained in the unit circle \cite[Lem.~1.21]{Agler1}, hence so is the spectrum of any operator similar to an invertible $m$-isometry. Below we show that, in a sense, the converse result is also true in finite-dimensional spaces.

\begin{thm}
\label{thm:m_isometries_finite_dim}
  Let $ \h $ be a finite-dimensional Hilbert space and let $T\in \B(\h) $. Let $m \geq 1$ be odd. The following are equivalent.
  \begin{enumerate}
\item $T$ is similar to an $m$-isometry.

\item There is an equivalent Hilbert norm with respect to which $T$ is an $m$-isometry.

\item $\sup_{n\in \Z_{\neq 0}} \frac{\| T^{n}\|^{2}}{|n|^{m-1}} < \infty$.
\item $\sigma(T) \subset \mathbb{T} $ and the maximum Jordan block size of $T$ is at most $\frac{m+1}{2}$.
  \end{enumerate}
\end{thm}

\begin{proof}
(a) $\Leftrightarrow$ (b) This equivalence follows from two facts.
First, every equivalent Hilbert norm on $\h$ is of the form $\newnorm{\vec{x}}=\norm{A\vec{x}}$
for a positive invertible $A \in \B(\h)$.  Second, if $T = Q^{-1}SQ$ and $Q = UA$, in which $U$ is unitary and $A$ is positive and invertible, then $T = P^{-1}(U^*SU)P$.  Since $m$-isometries are preserved by unitary equivalence, the desired result follows from a straightforward use of \eqref{eq:NormCondition}.

\smallskip\noindent(a) $\Rightarrow$ (c) Suppose that $T$ is similar to an $m$-isometry.  
Then $\sigma(T) \subset \T$ \cite[Lem.~1.21]{Agler1}, so $T$ is invertible.
The proof of \cite[Prop.~1.5]{Agler1} ensures that $\norm{T^n}^2/n^{m-1}$ is bounded as $n \to \infty$.
Since \eqref{eq:InverseAlso} ensures that $T^{-1}$ is similar to an $m$-isometry, we obtain (c).

\smallskip\noindent(c) $\Rightarrow$ (d)
Since $ r(T) = \lim_{n\to \infty} \sqrt[n]{\| T^{n}\|}$, we infer from (c) that $ r(T) = r(T^{-1}) = 1 $. Using the fact that $\sigma(T^{-1}) = \{ 1/ \lambda : \lambda \in \sigma(T) \}$, it follows that $\sigma(T) \subseteq \T$. 
Suppose that the largest Jordan block of $T$
is $J_k(\lambda)$ and $\vec{v}\in \h$ is a generalized eigenvector of type $k$ \cite[p. 357]{Bronson}. 
Then \cite[Sec. 9.7 and 9.8]{Bronson} ensure that for every $n\geq 1$, 
\begin{equation*}
    T^n \vec{v} = \sum_{j=0}^{k-1} \binom{n}{j}\lambda^{n-j}(T-\lambda I)^j \vec{v}.
\end{equation*}
Thus, $\norm{T^{n}\vec{v}}^{2}$ is a polynomial in $n$ of degree $2k-2$. Condition (c) ensures that $2k-2 \leq m-1$; that is, $k \leq \frac{m+1}{2}$.

\smallskip\noindent (d) $\Rightarrow$ (a)
Without loss of generality, suppose that $\h = \C^{n}$ endowed with the standard inner product. Then
$T = P^{-1}JP$, in which $P \in \M_n$ is invertible and 
$J = J_1 \oplus J_2 \oplus \cdots \oplus J_p$ with each
$J_i$ a Jordan block for some $\lambda \in \sigma(T)$.
Condition (d) ensures that $T$
is similar to $V + N$, in which $V$ is unitary, $N$ is nilpotent of order at most $\frac{m+1}{2}$, and $VN = NV$. Thus, $T$ is $m$-isometric by \cite[Thm.~2.2]{Bermudez}.
\end{proof}

The previous result does not hold if $\h$ is infinite dimensional.

\begin{exem}
  Let $ m\geq 2 $ and define
  \begin{equation*}
    \omega_{n} = 
    \begin{cases}
      1 & \text{if $n\leq 1$},\\
      \frac{n^{m-1}}{(n-1)^{m-1}} & \text{if $n \geq 2$}.
    \end{cases}
  \end{equation*}
  Let $T \in \B(\ell^{2}(\Z))$ be the bilateral weighted shift with weights $\omega_n$; that is,
  $T\vec{e}_{n} = \omega_{n+1}\vec{e}_{n+1}$ for $n\in \Z$.  Since $\omega_n \to 1$ as $n \to \pm \infty$, it follows that 
  $\sigma(T) = \T$ \cite[Prop.~15 \& Thm.~5]{Shields}.  In particular, $T$ is invertible.
  For $p \geq 1$, observe that
  \begin{equation*}
    T^{p} \vec{e}_{n} = 
    \begin{cases}
      \frac{(n+p)^{m-1}}{n^{m-1}}\vec{e}_{n+1} & \text{for $n\geq 1$},\\
      (n+p)^{m-1}\vec{e}_{n+1} & \text{for $-p+1 < n < 1$},\\
      \vec{e}_{n+1} & \text{for $n\leq -p+1$}.
    \end{cases}
  \end{equation*}
  Thus,
  \begin{equation}\label{ExFormPolyPositivePowers}
    \| T^{p}\|^{2} = (1+p)^{m-1}.
  \end{equation}
  Since $\omega_{n} \geq 1$ for each $n\in\Z$, we have $\norm{T\vec{x}} \geq \norm{\vec{x}}$ for $\vec{x}\in \ell^{2}(\Z)$. Hence $\norm{T^{-p}} \leq 1$ for $p\geq 1$. 
  Combining this with \eqref{ExFormPolyPositivePowers} shows that condition (c) of Theorem \ref{thm:m_isometries_finite_dim} holds.
  
  Suppose toward a contradiction that there is an $m$-isometry $V\in \B(\h)$ such that $V = STS^{-1}$ for some invertible $S \in \B(\h)$. Since $T$ is invertible, so is $V$ and hence $V^{-1}$ is $m$-isometric; see \eqref{eq:InverseAlso}. 
  For every $\vec{u}\in \ell^{2}(\Z)$ there is a polynomial $p_{\vec{u}}\in \mathbb{R}[x] $ of degree at most $m-1$ such that $p_{\vec{u}}(n) = \norm{ V^{-n}S\vec{u} }^{2}$ for all
  $n\geq 0$ \cite[p.~389]{Agler1}. Then for all $\vec{u} \in \ell^2(\Z)$ and $n \geq 0$,
  \begin{equation*}
    p_{\vec{u}}(n) = \norm{ V^{-n}S\vec{u} }^{2} = \norm{ST^{-n}\vec{u} }^{2} \leq \norm{ S }^{2} \norm{ \vec{u}}^{2}.
  \end{equation*}
  The polynomial $ p_{\vec{u}}$ is bounded and hence constant. Thus, $\norm{V^{-1}S\vec{u}}^{2} = p_{\vec{u}}(1) = p_{\vec{u}}(0) = \norm{ S\vec{u}}^{2}$ for all $\vec{u}\in \ell^{2}(\Z)$, which means that $V^{-1}$ isometric.  Since $V$ is also invertible, it is unitary.
  However, \eqref{ExFormPolyPositivePowers} ensures that
  \begin{equation*}
    (1+n)^{m-1} = \lVert T^{n}\rVert^{2}\leq \lVert S\rVert^{2}\lVert S^{-1}\rVert^{2}\lVert V^{n}\rVert^{2} = \lVert S\rVert^{2}\lVert S^{-1}\rVert^{2}
  \end{equation*}
  for all $n \geq 1$, which contradicts the assumption that $m \geq 2$.
\end{exem}

The main result of this section is the following theorem, which describes strict 
$m$-isometries on finite-dimensional Hilbert spaces.

\begin{thm}\label{thm:similar_strict_m}
Let $\h$ be an $n$-dimensional Hilbert space and
let $m$ be an odd integer such that $3 \leq m \leq 2n-1$. Then 
$T \in \B(\h)$ is similar to a strict $m$-isometry if and only if
$T = U + Q$, in which
$U$ is diagonalizable with at most $\frac{2n+1-m}{2}$ distinct unimodular eigenvalues,
$Q$ is nilpotent of order $\frac{m+1}{2}$, and $UQ = QU$.
\end{thm}

\begin{proof}
\noindent $(\Leftarrow)$  This follows from \cite[Thm.~2.2]{Bermudez}.

\noindent$(\Rightarrow)$ 
Suppose that $T \in \B(\h)$ is similar to a strict $m$-isometry.
On a finite-dimensional Hilbert space, strict $m$-isometries with $m$ odd are precisely those operators of the form $U+Q$, in which $U$ is unitary, $Q$ is nilpotent of order $\frac{m+1}{2}$, and $UQ = QU$ \cite[Thm.~2.7]{Bermudez}.
Thus, there exists an invertible $P$, unitary $U$, and nilpotent $Q$ of order $\frac{m+1}{2}$ such that
$T = PUP^{-1} + PQP^{-1}$ and $UQ = QU$.
Then $PUP^{-1}$ is diagonalizable with unimodular eigenvalues, $PQP^{-1}$ is nilpotent of order $\frac{m+1}{2}$, and $PUP^{-1}$ and $PQP^{-1}$ commute.
Suppose that $T$ has $r$ distinct eigenvalues.
Aside from the largest Jordan block, the remaining $r-1$ eigenvalues account for at least $r-1$ dimensions. Therefore, the largest Jordan block has size at most $n-(r-1)$.
If $r \geq \frac{2n+3-m}{2}$, then $n-(r-1) \leq \frac{m-1}{2}<\frac{m+1}{2}$, which contradicts \cite[Thm.~2.7]{Bermudez}.
\end{proof}

\begin{coro}
Let $\h$ be an $n$-dimensional Hilbert space.
Then $T \in \B(\h)$ is similar to a strict $(2n-1)$-isometry if and only if $T = \alpha I + Q$,
in which $\alpha \in \T$ and $Q$ is nilpotent of order $n$.
\end{coro}

\begin{proof}
    Let $m = 2n-1$ above and deduce that $U$ has exactly one eigenvalue.
\end{proof}

For each $A, B \in \B(\h)$, let 
$L_{A}, R_{B} \in \B(\B(\h))$ denote left multiplication by $A$ and right multiplication by $B$, respectively, defined for $X \in \B(\h)$ by
\begin{equation*}
L_{A}(X) = A X 
\qquad \text{and} \qquad 
R_{B}(X) = X B.   
\end{equation*}   
For $m \in \mathbb{Z}_{\geq 1}$, define
\begin{align*}
\Delta_{A^{*}, B}^{m}(X)
&= (L_{A^{*}} R_{B} - I)^{m}(X) \\
&= \bigg( \sum_{k=0}^{m} (-1)^{k} \binom{m}{k} (L_{A^{*}} R_{B})^{m-k} \bigg)(X) \\
&= \sum_{k=0}^{m} (-1)^{k} \binom{m}{k} \, A^{*(m-k)} X B^{m-k}.
\end{align*}
Then $\Delta^1_{T^*,T} = (-1)^0 T^*T + (-1)^1 I = T^*T - I$ and
\begin{equation*}
\beta_{m}(T) = (-1)^{m} \Delta_{T^{*}, T}^{m}(I).    
\end{equation*}   

The next theorem explores when an operator is similar to an $m$-partial isometry. 

\begin{thm}\label{sim-m-partialisom} 
Let $T \in \B(\h)$ and $m \in \mathbb{N}$. Then $T$ is similar to an $m$-partial isometry if and only if there exists an $S \in \B(\h)$ such that $T\Delta_{S,T}^{m}(I)=0$ and $T^* \sim_{+} S$.
\end{thm}

\begin{proof} 
$(\Rightarrow)$ Suppose that $V = A^{-1}TA$ for some $m$-partial isometry $V$. Then 
\begin{align*}
0&= \sum_{k=0}^{ m }(-1)^{k}\binom{m}{k}V V^{* m-k} V^{m-k}\\
&= A^{-1}\Bigg (\sum_{k=0}^{ m }(-1)^{k}\binom{m}{k} T A A^{*} T^{* m-k}(A^{-1})^{*} A^{-1}T^{m-k}\Bigg )A,
\end{align*}
so, 
\begin{equation*}
T\sum_{k=0}^{ m }(-1)^{k}\binom{m}{k}  (A A^{*}) T^{* m-k}(A A^*)^{-1}T^{m-k}=0.     
\end{equation*}
Then $S=(AA^*) T^*(AA^*)^{-1}$ satisfies $T^* \sim_+ S$ and
\begin{equation*}
    \sum_{k=0}^{ m }(-1)^{k}\binom{m}{k} TS^{ m-k} T^{m-k}=0.
\end{equation*}

\noindent$(\Leftarrow)$ Suppose that $T^{*}=P S P^{-1}$ for some $P>0$ and 
\begin{equation*}
T\triangle_{S,T}^{m}(I) = T \sum_{k=0}^{ m }(-1)^{k}\binom{m}{k} S^{ m-k} T^{m-k}=0.
\end{equation*}
Then
\begin{equation*}
T\sum_{k=0}^{ m }(-1)^{k}\binom{m}{k} P^{-1}T^{*{m-k}}P T^{m-k}=0.
\end{equation*}
Let $P^{1 / 2}$ be the positive square root of $P$ and let $V=P^{1 / 2} T P^{-1 / 2}$, so that
\begin{align*}
& V\sum_{k=0}^{ m }(-1)^{k}\binom{m}{k} V^{* m-k} V^{m-k}\\
&= P^{1 / 2}\Bigg( \sum_{k=0}^{ m }(-1)^{k}\binom{m}{k} T P^{-1 / 2} P^{-1 / 2} T^{* m-k}P^{1 / 2} P^{1 / 2}T^{m-k}\Bigg)P^{-1 / 2}\\
&= 0. 
\end{align*}
Thus, $V$ is an $m$-partial isometry that is similar to $T$. 
\end{proof}

The following theorem generalizes \cite[Thm.~3.1]{Badea} and \cite[Thm.~3.1]{Mbekhta}.

\begin{thm}\label{sim-m-partialisom2}
Let $T\in\mathcal B(\h)$ and let $m\in\mathbb N$.
Then $T$ is similar to an $m$-partial isometry whose adjoint is also an
$m$-partial isometry if and only if there exists an equivalent Hilbertian
norm $\newnorm{\cdot}$ on $\h$ such that if $T^{\sharp}$ denotes the
adjoint of $T$ with respect to the inner product associated with this norm,
then
\begin{equation*}
T\Delta^m_{T^{\sharp},T}(I)=0
\qquad\text{and}\qquad
T^{\sharp}\Delta^m_{T,T^{\sharp}}(I)=0.
\end{equation*}
If $m$ is odd, then $T$ and $T^{\sharp}$ are $m$-concave
with respect to $\newnorm{\cdot}$; that is, for all $\vec{x} \in \h$,
\begin{equation*}
\sum_{k=0}^{m}(-1)^{m-k}\binom{m}{k}\newnorm{T^{k}\vec{x}}^{2}\leq 0
\quad \text{and}\quad
\sum_{k=0}^{m}(-1)^{m-k}\binom{m}{k}\newnorm{(T^{\sharp})^{k} \vec{x}}^{2}\leq 0.
\end{equation*}
\end{thm}

\begin{proof}
$(\Rightarrow)$
Suppose that $T$ is similar to an $m$-partial isometry $V$ whose adjoint $V^*$ is also an $m$-partial isometry. Then there exists an invertible $X\in\B(\h)$ such that $V=XTX^{-1}$.
Define an inner product on $\h$ by
\begin{equation*}
\newinner{ \vec{u}, \vec{v} }
=
\langle X\vec{u},X\vec{v}\rangle,
\end{equation*}
for $\vec{u},\vec{v}\in\h$; the associated norm is
$\newnorm{\vec{u}} = \newinner{ \vec{u},\vec{u}}^{1/2} = \norm{X \vec{u}}$.
Since $X^*X$ is positive and invertible, $\newnorm{\cdot}$
is a Hilbertian norm equivalent to the original one.
Let $T^{\sharp}=X^{-1}V^*X$ denote the adjoint of $T$ with respect to this new inner product.

Since $V$ is an $m$-partial isometry, $V\beta_m(V)=0$.
Moreover, $X$ is a unitary operator from
$(\h,\newinner{\cdot,\cdot})$ onto
$(\h,\langle\cdot,\cdot\rangle)$, so
\begin{equation*}
\beta_m^{\sharp}(T)=X^{-1}\beta_m(V)X,
\end{equation*}
in which
\begin{equation*}
\beta_m^{\sharp}(T)
=
\sum_{k=0}^m(-1)^{m-k}\binom{m}{k}(T^{\sharp})^kT^k.
\end{equation*}
Therefore,
\begin{equation*}
T\beta_m^{\sharp}(T)
=
X^{-1}V\beta_m(V)X
=
0.
\end{equation*}
Since
\begin{equation*}
\beta_m^{\sharp}(T)
=
\sum_{k=0}^m(-1)^{m-k}\binom{m}{k}(T^{\sharp})^kT^k
=
\Delta^m_{T^{\sharp},T}(I),
\end{equation*}
we obtain
\begin{equation*}
T\Delta^m_{T^{\sharp},T}(I)=0.
\end{equation*}
Similarly, since $V^*$ is an $m$-partial isometry, 
$V^*\beta_m(V^*)=0$.
Since $T^{\sharp}=X^{-1}V^*X$, it follows that
$T^{\sharp}\beta_m^{\sharp}(T^{\sharp})
=X^{-1}V^*\beta_m(V^*)X
=0$.
Moreover, since $(T^{\sharp})^\sharp=T$, 
\begin{equation*}
\beta_m^{\sharp}(T^{\sharp})
=
\sum_{k=0}^m(-1)^{m-k}\binom{m}{k}T^k(T^{\sharp})^k
=
\Delta^m_{T,T^{\sharp}}(I).
\end{equation*}
Consequently, $T^{\sharp}\Delta^m_{T,T^{\sharp}}(I)=0$.

If $m$ is odd, then the $m$-concavity of $T$ follows from
Lemma \ref{Lemma:Projection} applied to $T$ on the Hilbert space $(\h,\newinner{\cdot,\cdot})$.
Similarly, since $T^{\sharp}$ is also an $m$-partial isometry with respect
to $\newnorm{\cdot}$, Lemma \ref{Lemma:Projection} applied to $T^{\sharp}$ gives the $m$-concavity
of $T^{\sharp}$.

\noindent($\Leftarrow)$
Suppose that there exists an equivalent Hilbertian norm
$\newnorm{\cdot}$ on $\h$ such that
\begin{equation*}
T\Delta^m_{T^{\sharp},T}(I)=0
\qquad\text{and}\qquad
T^{\sharp}\Delta^m_{T,T^{\sharp}}(I)=0,
\end{equation*}
in which $T^{\sharp}$ denotes the adjoint of $T$ with respect to this inner product associated with this norm. Since
\begin{equation*}
\Delta^m_{T^{\sharp},T}(I)
=
\sum_{k=0}^m(-1)^{m-k}\binom{m}{k}(T^{\sharp})^kT^k
=
\beta_m^{\sharp}(T),
\end{equation*}
we get $T\beta_m^{\sharp}(T)=0$.
Thus, $T$ is an $m$-partial isometry on $(\h,\newnorm{\cdot})$.
Since $(T^{\sharp})^\sharp=T$, 
\begin{equation*}
\Delta^m_{T,T^{\sharp}}(I)
=\sum_{k=0}^m(-1)^{m-k}\binom{m}{k}T^k(T^{\sharp})^k
= \beta_m^{\sharp}(T^{\sharp}).
\end{equation*}
The identities
\begin{equation*}
T^{\sharp}\Delta^m_{T,T^{\sharp}}(I)=0
\quad\text{and}\quad
T^{\sharp}\beta_m^{\sharp}(T^{\sharp})=0
\end{equation*}
are therefore equivalent.
Thus, $T^{\sharp}$ is also an $m$-partial isometry on
$(\h,\newnorm{\cdot})$.

Since $\newnorm{\cdot}$ is an equivalent Hilbertian norm, there exists a positive invertible $G\in\B(\h)$ such that
\begin{equation*}
\newinner{ \vec{x},\vec{y} }
=
\langle G\vec{x},\vec{y}\rangle
\end{equation*}
for all $\vec{x},\vec{y}\in \h$. Let
$X=G^{1/2}$.
Then $X$ is a unitary operator from
$(\h,\newinner{\cdot,\cdot})$ onto the original
Hilbert space $(\h,\langle\cdot,\cdot\rangle)$.
Then $V=XTX^{-1}$ is similar to $T$. Since $T$ is an $m$-partial isometry on
$(\h,\newnorm{\cdot})$, unitary equivalence ensures that $V$ is an
$m$-partial isometry on the original Hilbert space.
Furthermore,
$V^*=XT^{\sharp} X^{-1}$.
Since $T^{\sharp}$ is an $m$-partial isometry on
$(\h,\newnorm{\cdot})$, unitary equivalence ensures that $V^*$
is an $m$-partial isometry on the original Hilbert space.
\end{proof}

\section{Further research}\label{Section:Further}
There are many potential avenues for additional investigation on the subject of $m$-partial isometries. We conclude this paper with several open questions that we hope will motivate further research.

\begin{Ques} 
Let $T$ be an $m$-partial isometry and let 
\begin{equation*}
S=\sum_{k=1}^m (-1)^{k-1}\binom{m}{k}T^{*k}T^{k-1}    
\end{equation*} 
be the generalized inverse of $T$ appearing in Lemma~\ref{lem:pseudompartial}. For $m=1$, this reduces to $S=T^*$, which is again a partial isometry whenever $T$ is. For $m\geq 2$, when is $S$ necessarily an $m$-partial isometry? 
Similar to an $m$-partial isometry?
\end{Ques}

\begin{Ques}
    In \cite{Mbekhta} Mbekhta proved that that $ T $ is a partial isometry if and only if it has a contractive generalized inverse. Does there exist a similar characterization of $m$-partial isometries?
\end{Ques}

In \cite{Badea} the authors considered so-called \emph{regular} operators; that is, operators $T\in \B(\h)$ with closed range such that
\begin{equation*}
    \N(T) \subset \bigcap_{n\geq 1}\R(T^n).
\end{equation*}
They showed that regular partial isometries are precisely direct sums of isometries, unitaries, and co-isometries.

\begin{Ques}
    Do regular $m$-partial isometries have a decomposition as the direct sum of an $m$-isometry, invertible $m$-isometry, and the adjoint of $m$-isometry?
\end{Ques}

\begin{Ques}\label{Question:Dilation}
    Does every $m$-partial isometry $T$ possess an $m$-isometric dilation if we assume $ r(T)\leq 1 $?
\end{Ques}

Garcia and Sherman proved that 
$A \in \MM_{n}$ is similar to a partial isometry if and only if 
(a) $\sigma(A) \subseteq \D^{-}$;
(b) if $\zeta \in \sigma(A) \cap \T$, then its algebraic and geometric multiplicities are equal;
and (c) $\dim \N(A) \geq \dim \N(A-\lambda I)$ for each $\lambda \in \sigma(A) \cap \D$ \cite{Garcia}.  Does an analogous characterization exist for $m$-partial isometries?

\begin{Ques}
Can $m$-partial isometries on finite-dimensional Hilbert spaces be concretely classified up to similarity?
\end{Ques}

The condition $ T\beta_{2}(T) = 0 $ can be written  as
\begin{equation*}
  T(2T^*-T^{*2}T)T = T,
\end{equation*}
so $S = 2T^*-T^{*2}T $ is a generalized inverse for $T$ and $ST$ is the orthogonal projection onto $\R(T^*)$. 
Moreover, $STS = S$, so $S$ is almost the Moore-Penrose inverse of $T$.  However, $TS$ need not be selfadjoint: if
\begin{equation*}
T = \begin{bmatrix} a & 0 \\ 1 & 0 \end{bmatrix}    
\end{equation*}
with $|a|^2 = \frac{1+\sqrt{5}}{2}$, then $T$ is a $2$-partial isometry and
\begin{equation*}
T S = \begin{bmatrix} -1 & 2a \\[2mm] \overline{a}\,\frac{1-\sqrt{5}}{2} & 2 \end{bmatrix}    
\end{equation*}
is not selfadjoint. All is not lost, however.  
The term ``selfadjoint'' above refers to the standard
inner product on $\C^2$.  Since different inner products give rise to different adjoints, the next question is natural.

\begin{Ques}
If $T$ is a $2$-partial isometry, is $2T^*-T^{*2}T$ the Moore--Penrose inverse of $T$ 
with respect to some inner product?
\end{Ques}

\bibliography{main}
\bibliographystyle{amsplain}

\end{document}